%% file: root.tex
\documentclass[letterpaper, 10 pt, conference]{ieeeconf}  

\IEEEoverridecommandlockouts                              

\usepackage{graphics} 
\usepackage{epsfig} 
\usepackage{mathptmx} 
\usepackage{times} 
\usepackage{amsmath} 
\usepackage{amssymb}  
\usepackage{amsmath,amsfonts,mathrsfs,latexsym,amscd}
\usepackage{mathtools}
\usepackage{physics}
\usepackage[ruled]{algorithm2e}
\usepackage{graphicx}

\input{preamble}

\title{\LARGE \bf
Tikhonov regularized exterior penalty dynamics for constrained variational inequalities
}

\author{Mathias Staudigl and Siqi Qu$^{1}$
\thanks{*This work was supported by the CSC and benefited from the support of the FMJH Program Gaspard Monge for optimization and operations research and their interactions with data science.}
\thanks{$^{1}$ University of Mannheim, Department of Mathematics, B6, 
        {\tt\small mathias.staudigl@uni-mannheim.de,  \tt\small qu.siqi@uni-mannheim.de}}
        }%

\begin{document}

\maketitle
\thispagestyle{empty}
\pagestyle{empty}

\begin{abstract}
Solving equilibrium problems under constraints is an important problem in optimization and optimal control. In this context an important practical challenge is the efficient incorporation of constraints. We develop a continuous-time method for solving constrained variational inequalities based on a new penalty regulated dynamical system in a general potentially infinite-dimensional Hilbert space. In order to obtain strong convergence of the issued trajectory of our method, we incorporate an explicit Tikhonov regularization parameter in our method, leading to a class of time-varying monotone inclusion problems featuring multiscale aspects. Besides strong convergence, we illustrate the practical efficiency of our developed method in solving constrained min-max problems. 
\end{abstract}


\section{INTRODUCTION}
Let $\scrH$ be a real Hilbert space,  $\opA:\scrH\to2^{\scrH}$ a general maximally monotone operator, $\opD:\scrH\to\scrH$ a Lipschitz continuous and monotone operator, and $\setC\subseteq\scrH$ a closed convex set. Denote by $\NC_{\setC}(x)$ the outward normal cone to $\setC$. In this paper we are concerned with the study of a class of splitting methods for solving constrained variational equalities (VIs) of the form 
\begin{equation}\label{eq:MI}\tag{P}
0\in\Phi(x)= \opA(x)+\opD (x)+\NC_{\setC}(x).
\end{equation}
Problems of this form are ubiquitous in control and engineering. Important examples include inverse problems \cite{Bot:2014ab}, generalized Nash equilibrium problems \cite{Borgens:2018aa,BorKan21,YiPav19}, domain decomposition for PDEs \cite{Attouch:2010aa,Attouch:2008aa}, and many more. Motivated by these applications, we present an operator splitting method designed to approach a specific solution of \eqref{eq:MI} in a Hilbertian framework using a new dynamical system featuring multiscale aspects. Our construction relies on the assumption that the set constraint $\setC$ can be represented as the zero of another monotone operator $\opB:\scrH\to\scrH$, so that $\setC=\{x\in\scrH\vert \opB(x)=0\}$. While this setup might appear restrictive, it is in fact an almost universal constellation in equilibrium problems and convex optimization \cite{AminiYous19,Yousefian:2021aa,Benenati:2023aa}. Motivated by solving variational inequalities arising in game theory, we construct a dynamical system that exhibits strong convergence properties under mere monotonicity assumptions on the Lipschitzian operator $\opD$. This technical achievement allows us to approach large classes of generalized Nash equilibrium problems, without strong monotonicity assumption, as are commonly used in the perceived literature (see e.g. \cite{YiPav19,Sun:2021ab,Bianchi:2022aa}). In particular, our method allows us to apply the developed scheme directly to convex-concave saddle-point problems, which regained a lot of importance in the data science and machine learning community due to its important applications in reinforcement learning \cite{Omidshafiei:2017aa}, and generative adversarial networks \cite{Goodfellow:2020aa}. However, our approach applies to many more important instances of constrained VIs, which are of particular relevance in Generalized Nash equilibrium problems (GNEP).
\begin{example}[Generalized Nash equilibrium Problems]
The GNEP is a very important class of multi-agent optimization problems in which non-cooperatively agents aim to minimize their private cost functions, given the other agents' decisions and eventually also joint coupling constraints restricting the feasible set for each player. It has attracted enormous attention within the systems and control community as a general mathematical template to design distributed control strategies in complex networked systems \cite{Facchinei:2010aa,Grammatico:2017aa,Franci:2020aa,Cui:2021aa}. Consider $N$ agents where the local optimization problem of agent $i$ reads as  
$
\min_{x_{i}} J_{i}(x_{i},x_{-i})=\ell_{i}(x_{i},x_{-i})+r_{i}(x_{i})\quad \text{s.t.: } (x_{i},x_{-i})\in\setC.
$
The constraint $x=(x_{i},x_{-i})\in\setC$ represents a coupling constraint and for simplicity we assume this set of to be described by linear restrictions in the control variables of the agents: Denote by $\scrH=\prod_{i}\scrH_{i}$, we assume that the jointly feasible set of the game is described as $\setC=\{x\in\scrH\vert \sum_{i=1}^{N}T_{i}x_{i}=b\}$. Here $T_{i}:\scrH_{i}\to\scrG$ are bounded linear operators between two real Hilbert spaces $\scrH_{i}$ and $\scrG$, and $b\in\scrG$ is given. Assuming that the cost functions $\ell_{i}(\cdot,x_{-i})$ and $r_{i}(\cdot)$ are convex in the own control variable, we arrive at jointly convex GNEP for which we can use variational inequality methods to compute Nash equilibrium points. In this paper we follow penalty techniques to enforce the joint feasibility restrictions $x\in\setC$ \cite{facchinei2010penalty,sun2020continuous}. Consider the function $g(x)=\frac{1}{2}\norm{\sum_{i=1}^{N}T_{i}x_{i}-b}^{2}$ and $\opB(x)=\nabla g(x)$. Then, it is clear that $\setC=\zer(\opB)=\argmin_{x\in\scrH}g(x)$. Assuming additionally that the functions $\ell_{i}(\cdot,x_{-i})$ are continuously differentiable for all $x_{-i}$, we can define the operator $\opD(x):=[ \nabla_{x_{1}}\ell_{1}(x);\ldots;\nabla_{x_{N}}\ell_{N}(x)]$, as well as the set-valued operator  $\opA(x)=\partial r_{1}(x_{1})\times\cdots\times\partial r_{N}(x_{N})$. Finding a Nash equilibrium of the game can in this setting by casted as the search for a solution to the variational inequality \eqref{eq:MI} with the data just defined. \close
\end{example}
To deal with the constraints present in the general formulation \eqref{eq:MI}, we study the trajectories defined by the continuous-time dynamical system with components $(p,x)$, defined as 
\begin{equation}\label{eq:1storder}
\left\{\begin{split}
p(t)&=\res_{\lambda(t)A}(x(t)-\lambda(t)V(t,x(t)),\\
\dot{x}(t)&=p(t)-x(t)+\lambda(t)[V(t,x(t))-V(t,p(t))]
\end{split}\right.
\end{equation}
where $V(t,x):=\opD(x)+\eps(t)x+\beta(t)\opB(x)$ and $\lambda(t)>0$ acts as a step-size parameter. This dynamical system is a continuous-time version of Tseng's modified extragradient method \cite{Tseng:2000aa,Franci:2020aa}. The vector field defining this dynamics carries two parameters in its description. The addendum $\eps(t)x$ acts as a Tikhonov regularization on the trajectories, which will allow us to prove strong convergence (i.e. convergence in norm) to the least norm solution of the original problem \eqref{eq:MI}. The second addendum $\beta(t)\opB(x)$ acts as an exterior penalty to the method which forces the trajectory to move towards the set $\zer(\opB)=\setC$ over time. The main result (Theorem \ref{th:main}) of this note establishes the strong convergence of the trajectory $x(t)$ to the least norm solution of the underlying constrained VI \eqref{eq:MI} whenever the Tikhonov parameter $\eps(t)$ vanishes and the penalty parameter $\beta(t)$ grows to infinity, subject to some verifiable conditions. 
\subsection{Related Literature}
Continuous time methods for variational inequalities have received a lot of attention in the last decades due to their close link to iterative methods via time-discretization methods, and the availability of Lyapunov analysis to assess their asymptotic properties. The dynamical system \eqref{eq:1storder} is an example of a modified extragradient method, which received a lot of attention recently within the machine learning community. For our developments, an important feature of this dynamical system is the presence of multi-scale aspects, embodied by the functions $\eps(t)$ (Tikhonov) and $\beta(t)$ (Penalty). The impact of such dynamical features in a continuous-time system are well-documented within optimization \cite{AttCzarPey11,Peypouquet:2012aa,Bot:2014,Bot:2016aa}. At the same time, Tikhonov regularization is a classical tool in variational analysis, that has been studied in connection with variational inequalities in many papers (see e.g. \cite{Cominetti:2008aa,Bot:2020aa}). The contribution we are making here is to explicitly study the interplay of both effects simultaneously. We give verifiable conditions that guarantee asymptotic convergence the least norm solution of the original problem using our exterior penalty approach. While a seemingly natural approach, we could not identify a similar analysis in the literature, and therefore belief that this constitutes a new results. Besides theoretical convergence statements, we report on the empirical performance of the method in solving convex-concave saddle point problems. 

\section{PRELIMINARIES}
\subsection{Notation}
Let $\scrH$ be a real Hilbert space with inner product $\inner{\cdot,\cdot}$ and associated norm $\norm{\cdot}$. The norm ball with center $x$ and radius $\eps$ is denoted by $\B(x,\eps)$. The symbols $\wlim$ and $\to$ denote weak and strong convergence, respectively. 
%
For a closed convex set $\setC\subset\scrH$ we define the normal cone as 
$\NC_{\setC}(x)=\{u\in\scrH\vert \inner{y-x,u}\leq 0\}$ if $x\in\setC$, and $\NC_{\setC}(x)=\emptyset$ else. For a set-valued operator $\M:\scrH\to 2^{\scrH}$ we denote by $\gr(\M)=\{(x,u)\in\scrH\times\scrH\vert u\in\M(x)\}$ its graph, $\dom(\M)=\{x\in\scrH\vert\M(x)\neq\emptyset\}$ its domain, and by $\M^{-1}:\scrH\to 2^{\scrH}$ its inverse, defined by  $(u,x)\in \gr(\M^{-1})\iff (x,u)\in\gr(\M).$ We let $\zer(\M)=\{x\in\scrH\vert 0\in \M(x)\}$ denote the set of zeros of $\M$. An operator $\M$ is $c$-strongly monotone if $c\geq 0$ and $\inner{x-y,u-v}\geq c\norm{x-y}^{2}$ for all $(x,u),(y,v)\in\gr(\M)$. In case $c=0$ holds, the operator $\M$ is said to be monotone. A monotone operator $\M$ is maximally monotone if there exists no proper monotone extension of the graph of $\M$ on $\scrH\times\scrH$. 
\begin{fact}\cite[Proposition 23.39]{BauCom16}
If $\M$ is maximally monotone, then $\zer(\M)$ is a convex and closed set.
\end{fact}

\begin{fact}\label{fact:angle}
If $\M$ is maximally monotone, then $p\in\zer(\M)\iff \inner{u-p,w}\geq 0\quad\forall (u,w)\in\gr(\M).$
\end{fact}

The resolvent of $\M$, $\res_{\M}:\scrH\to 2^{\scrH}$ is defined by $\res_{\M}=(\Id+\M)^{-1}$. If $\M$ is maximally monotone, then $\res_{\M}$ is single-valued and maximally monotone. 
\subsection{Properties of perturbed solutions}
As already mentioned in the Introduction, our approach combines Tikhonov regularization and penalization. Therefore, we first have to understand the properties of sequences of solutions of the intermediate auxiliary problems. Specifically, let $(\eps,\beta)\in(0,\infty)\times(0,\infty)$ be a given pair of parameters. The trajectory of the dynamical system \eqref{eq:1storder} is going to track solutions of the auxiliary problem  
\begin{equation}\label{eq:MIauxiliary}
0\in \Phi_{\eps,\beta}(x)= (\opA+\opD+\eps\Id+\beta\opB)(x) 
\end{equation}
This VI involves the parametric family of mappings $V_{\eps,\beta}(x):=\opD(x)+\eps x+\beta\opB(x)$ for all $x\in\scrH$. 
\begin{assumption}\label{ass:data}
$\opA:\scrH\to 2^{\scrH}$ is maximally monotone. $\opD:\scrH\to\scrH$ is maximally monotone and $\frac{1}{\eta}$-Lipschitz continuous. $\opB:\scrH\to\scrH$ is maximally monotone, satisfies $\zer(\opB)=\setC$ and is $\frac{1}{\mu}$-Lipschitz.\close
\end{assumption}
Our first lemma collects basic regularity properties of this family of mappings. We skip the easy proof because of space limitations.
\begin{lemma}
For all $\eps,\beta>0$, we have 
\begin{itemize}
\item[(i)] $V_{\eps,\beta}$ is Lipschitz continuous with modulus $L_{\varepsilon,\beta}\eqdef \frac{1}{\eta}+\eps+\frac{\beta}{\mu}$;
\item[(ii)] If either $\dom(\opD)=\scrH$ or $\dom(\opD)\cap\Int\dom(\opB)\neq\emptyset$, then $V_{\eps,\beta}$ is maximally monotone and even strongly monotone. 
\end{itemize}
\end{lemma}
We see that the family of auxiliary problems \eqref{eq:MIauxiliary} becomes a sequence of strongly monotone inclusion problems. Therefore, for each parameter pair $(\eps,\beta)\in (0,\infty)\times[0,\infty)$, the set  $\zer(\Phi_{\eps,\beta})$ reduces to a singleton $\{\bar{x}(\eps,\beta)\}$. The next proposition establishes some consistency and regularity properties of this parametric family of solutions.
\begin{proposition}\label{prop:asymptotics}
Let $(\eps_{n})_{n\in\N},(\beta_{n})_{n\in\N}$ be sequences in $(0,\infty)$ such that $\eps_{n}\to 0,\beta_{n}\to+\infty$. Then $\bar{x}(\eps_{n},\beta_{n})\to z\in\zer(\Phi)$ satisfying $\norm{z}=\inf\{\norm{x}:x\in\zer(\Phi)\}$. 
\end{proposition}
We next prove that $(\eps,\beta)\mapsto \bar{x}(\eps,\beta)$ is a locally Lipschitz continuous function. 
\begin{proposition}\label{prop:solutionmap}
The solution mapping $(\eps,\beta)\mapsto \bar{x}(\eps,\beta)$ is locally Lipschitz continuous. In particular, for all $t_{1}=(\eps_{1},\beta_{1})$ and $t_{2}=(\eps_{2},\beta_{2})$, we have 
\begin{equation}
\norm{\bar{x}(t_{2})-\bar{x}(t_{1})}\leq \frac{\ell}{\eps_{1}}\left(\abs{\beta_{2}-\beta_{1}}+\abs{\eps_{2}-\eps_{1}}\right),
\end{equation}
where $\ell\eqdef \max\{\sup_{x\in\B(0,\ca)}\norm{\opB(x)},\ca\}$, and $\ca\eqdef\inf\{\norm{x}:x\in\zer(\Phi)\}$. 
\end{proposition}

\section{Penalty regulated dynamical systems}
As solutions of the proposed dynamical system \eqref{eq:1storder} we consider absolutely continuous functions  \cite{Sontag:2013aa}. Recall that a function $f:[0,b]\to\scrH$ (where $b>0$) is said to be absolutely continuous if there exists an integrable function $g:[0,b]\to\scrH$ such that 
$$
f(t)=f(0)+\int_{0}^{t}g(s)\dif s\quad\forall t\in[0,b].
$$
Given functions $\eps,\beta,\lambda:(0,\infty)\to(0,\infty)$, we set $V(t,x)\equiv V_{\eps(t),\beta(t)}(x)$, and introduce the reflection $R_{t}(x)\eqdef x-\lambda(t)V(t,x)$. Furthermore, define the vector field $f(t,x):[0,\infty)\times\scrH\to\scrH$ by
\begin{equation}\label{eq:Tseng}
f(t,x)\eqdef \left(R_{t}\circ \res_{\lambda(t)A}\circ R_{t}\right)(x)-R_{t}x
\end{equation}
The first-order dynamical system \eqref{eq:1storder} reads then compactly as 
\begin{equation}\label{eq:TsengDynamics}
\dot{x}(t)=f(t,x(t)),x(0)=x^{0}\in\scrH\text{ given. }
\end{equation}
We say that $x:[0,\infty)\to\scrH$ is a strong solution of \eqref{eq:TsengDynamics} if $t\mapsto x(t)$ is absolutely continuous on each interval $[0,b],0<b<\infty$ and  $\frac{\dif}{\dif t}x(t)=f(t,x(t))$ for almost every $t\in[0,+\infty)$. 
%
\subsection{Existence of solutions}
To prove existence and uniqueness of strong global solutions to \eqref{eq:TsengDynamics}, we use the Cauchy-Lipschitz theorem for absolutely continuous trajectories \cite{Sontag:2013aa}. 
 \begin{assumption}\label{ass:parametercontinuous}
The functions $\lambda,\eps,\beta:[0,\infty)\to(0,\infty)$ are continuous on each interval $[0,b],0<b<\infty$. Additionally, $\eps(\cdot),\beta(\cdot)$ are continuously differentiable. The mapping $t\mapsto\eps(t)$ is monotonically decreasing and $t\mapsto\beta(t)$ is monotonically increasing. 
\end{assumption}
The following Lemma establishes Lipschitz continuity of the vector field $f(t,x)$. The proof is similar to \cite[Lemma 5.1]{Bot:2020aa}, and thus omitted due to space restrictions.
\begin{lemma}
Assume that  $\lambda(t)\eps(t)<\frac{1}{\eta}+\frac{\beta(t)}{\mu}$ for all $t\geq 0$. For all $x,y\in\scrH$ and $t\geq 0$, we then have 
$$
\norm{f(t,x)-f(t,y)}\leq\kappa(t)\norm{x-y}
$$
where $\kappa(t)\equiv \sqrt{(1+2\lambda(t)L(t))\left(1+\lambda^{2}(t)L^{2}(t)-2\lambda(t)\eps(t)\right)}$ and $L(t)\equiv L_{\eps(t),\beta(t)}=\frac{1}{\eta}+\eps(t)+\frac{\beta(t)}{\mu}$. 
\end{lemma}
Next, we establish a growth bound on the vector field. To that end, we define parameterized maps
$R_{\lambda,\eps,\beta}(x)\eqdef x-\lambda V_{\eps,\beta}(x)$ and $F_{\lambda,\eps,\beta}(x)\eqdef \left(R_{\lambda,\eps,\beta}\circ \res_{\lambda A}\circ R_{\lambda,\eps,\beta}\right)(x)-R_{\lambda,\eps,\beta}x$. Obviously, we have $f(t,x)\equiv F_{\lambda(t),\eps(t),\beta(t)}(x)$ and $R_{t}(x)\equiv R_{\lambda(t),\eps(t),\beta(t)}(x)$. Because of space limitations we omit the fairly simple proof of the next result.
\begin{lemma}\label{lem:Fbound}
For all $(\lambda,\eps,\beta)\in\R^{3}_{+}$ satisfying $\lambda<\frac{1}{1/\eta+\eps+\beta/\mu}$, and $x\in\scrH$, there exists a universal constant $C>0$ such that 
$\norm{F_{\lambda,\eps,\beta}(x)}\leq C(1+\norm{x}).$
\end{lemma}
%
Existence and uniqueness of solutions to the dynamical system \eqref{eq:TsengDynamics} is now a straightforward consequence of the Picard-Lindelöf Theorem. 

\subsection{Convergence of trajectories}
In order to show the convergence, we start with some technical lemmata. To simplify the notation, we denote by $\bar{x}(t)=\bar{x}(\eps(t),\beta(t))$ the path of unique elements of $\zer(\Phi_t)$ where $\Phi_t\equiv \Phi_{\varepsilon(t),\beta(t)}$.
\lemma
\label{Lemma 3.7} 
	For almost all $t \in [0,+\infty)$, we have
	\begin{align*}
\norm{x(t)-p(t)}^2&-\norm{x(t)-\bar{x}(t)}^2\leq (1+2\lambda(t)\eps(t))\|p(t)-\bar{x}(t)\|^2\\
		&+2\lambda(t)\inner{ V(t,p(t))-V(t,x(t)),p(t)-\bar{x}(t)}
	\end{align*}
\begin{proof}
From \eqref{eq:1storder}, we have
	$
	(\Id+\lambda(t)A)p(t) \ni x(t)-\lambda(t)V(t,x(t)).
	$
	It follows,
	\begin{align*}
	\Phi_{t}(p(t))&=A(p(t))+V(t,p(t)) \\
	& \ni \frac{x(t)-p(t)}{\lambda(t)}-V(t,x(t))+V(t,p(t))=-\frac{\dot{x}(t)}{\lambda(t)}.
	\end{align*}
	By Assumption \ref{ass:data}, the operators $\opD$ and $\opB$ are maximally monotone, which implies that $\Phi_{t}$ is $\eps(t)$-strongly monotone. Consequently,
	\begin{equation}\label{eq:Phi-Monotone}
	\langle-\frac{\dot{x}(t)}{\lambda(t)}-0,p(t)-\bar{x}(t)\rangle\geq \varepsilon(t)\|p(t)-\bar{x}(t)\|^2.
	\end{equation}
	
	Using this property, we see
		\begin{align*}
			&2\lambda(t)\varepsilon(t)\norm{p(t)-\bar{x}(t)}^2\\
			&\leq 2\inner{ x(t)-p(t),p(t)-\bar{x}(t)}\\
			&~~~+2\lambda(t)\langle V(t,p(t))-V(t,x(t)),p(t)-\bar{x}(t) \rangle \\
			&=-\|x(t)-p(t)\|^2+\|x(t)-\bar{x}(t)\|^2-\|p(t)-\bar{x}(t)\|^2\\
			&~~~+2\lambda(t)\langle V(t,p(t))-V(t,x(t)),p(t)-\bar{x}(t) \rangle.
		\end{align*}
\end{proof}
\begin{lemma}\label{Lemma 3.8}
	Let $t \mapsto x(t)$ be the unique strong global solution of \eqref{eq:1storder} for almost all $t \in [0,+\infty)$. Then
	\begin{align*}
	\langle x(t)-\bar{x}(t),\dot{x}(t) \rangle \leq& (\lambda(t)L(t)-1)\|x(t)-p(t)\|^2\\
	&-\lambda(t)\varepsilon(t)\|p(t)-\bar{x}(t)\|^2
	\end{align*}
\end{lemma}
\begin{proof}
	For almost all $t \in [0,+\infty)$, we have
	\begin{align*}
		&2\langle x(t)-\bar{x}(t),\dot{x}(t) \rangle\\
		&= 2\langle x(t)-\bar{x}(t),p(t)-x(t) \rangle\\
		&~~~ +2\langle x(t)-\bar{x}(t),\lambda(t)[V(t,x(t))-V(t,p(t))] \rangle\\
		&=\|\bar{x}(t)-p(t)\|^2-\|\bar{x}(t)-x(t)\|^2-\|x(t)-p(t)\|^2\\
		&~~~ +2\lambda(t)\langle x(t)-\bar{x}(t),V(t,x(t))-V(t,p(t)) \rangle
	\end{align*}
	combining with Lemma \ref{Lemma 3.7}, we get
$\|\bar{x}(t)-p(t)\|^2-\|\bar{x}(t)-x(t)\|^2\leq-\|x(t)-p(t)\|^2-2\lambda(t)\varepsilon(t)\|p(t)-\bar{x}(t)\|^2+2\lambda(t)\inner{V(t,p(t))-V(t,x(t)),p(t)-\bar{x}(t)}.$ Hence,
	\begin{align*}
		&2\langle x(t)-\bar{x}(t),\dot{x}(t) \rangle\\
		&\leq -2\|x(t)-p(t)\|^2-2\lambda(t)\varepsilon(t)\|p(t)-\bar{x}(t)\|^2\\
		&~~~+2\lambda(t)\langle V(t,p(t))-V(t,x(t)),p(t)-\bar{x}(t) \rangle\\
		&~~~+2\lambda(t)\langle x(t)-\bar{x}(t),V(t,x(t))-V(t,p(t)) \rangle\\
		&\leq -2(1-\lambda(t)L(t))\|x(t)-p(t)\|^2-2\lambda(t)\varepsilon(t)\|p(t)-\bar{x}(t)\|^2
  	\end{align*}
	the proof is completed.
\end{proof}


\begin{theorem}
\label{th:main}
	Let $t \mapsto x(t)$ be the strong global solution of \eqref{eq:1storder}, where $t \geq 0$, for almost all $t \in [0,+\infty)$, assume that $\limsup_{\to\infty}\lambda(t)\beta(t)<\mu$, as well as that Assumptions \ref{ass:data} and \ref{ass:parametercontinuous} hold. Assume further that
	\begin{itemize}
		\item[(a)] $\lim_{t\to\infty}\delta(t)=0$ and $\lim\limits_{t \rightarrow +\infty}\int_{0}^{t}\delta(s)\dif s=\infty$, where $\delta(t)=\frac{1-\lambda(t)L(t)}{a^{2}(t)}$, and 
$a(t)=2+\frac{1}{\lambda(t)\eps(t)}+\frac{1}{\eta\eps(t)}+\frac{\beta(t)}{\mu\eps(t)}.$
		 \item[(b)] $\lim_{t\to\infty}\frac{\dot{\eps}(t)}{\eps(t)\delta(t)}=0$ and $\frac{\dot{\beta}(t)}{\eps(t)\delta(t)}=\scrO(1)$ as $t\to\infty$.
	\end{itemize}
Then $x(t)\rightarrow \Pi_{\zer(\Phi)}(0)$ as $t \rightarrow +\infty$.
\end{theorem}
\begin{proof}
	Define $\theta(t)=\frac{1}{2}\|x(t)-\bar{x}(t)\|^2$ where $t \geq 0$. From $\bar{x}(t)=\bar{x}(\varepsilon(t),\beta(t))$, we have
	$
	\dot{\theta}(t)=\langle x(t)-\bar{x}(t),\dot{x}(t)-\frac{\dd}{\dd t}\bar{x}(t)\rangle,
 	$
	where
	$
	\frac{\dd}{\dd t}\bar{x}(t)=\frac{\partial}{\partial \varepsilon}\bar{x}(\varepsilon(t),\beta(t))\dot{\varepsilon}(t)+\frac{\partial}{\partial \beta}\bar{x}(\varepsilon(t),\beta(t))\dot{\beta}(t).
	$
	Before proving the statement, we need to establish some estimates on $\frac{\dd}{\dd t}\bar{x}(t)$ first.
	
	(i) Estimate on $\frac{\partial}{\partial \varepsilon}\bar{x}(\varepsilon(t),\beta(t))$:
	
	Set $\varepsilon_1=\varepsilon, \varepsilon_2=\varepsilon+h$, where $h \rightarrow 0^+$, we obtain:
\begin{align*}
&\lim\limits_{h \rightarrow 0}\frac{\|\bar{x}(\varepsilon+h,\beta)-\bar{x}(\varepsilon,\beta)\|}{|h|}=\|\frac{\partial}{\partial \varepsilon} \bar{x}(\varepsilon,\beta)\| \\
&\leq \lim\limits_{h \rightarrow 0}\frac{1}{\varepsilon+h}\|\bar{x}(\varepsilon,\beta)\| = \frac{\|\bar{x}(\varepsilon,\beta)\|}{\varepsilon}.
\end{align*}

(ii) Estimate on $\frac{\partial}{\partial \beta}\bar{x}(\varepsilon(t),\beta(t))$: As in part (i) we obtain

$\lim\limits_{h \rightarrow 0}\frac{\|\bar{x}(\varepsilon,\beta+h)-\bar{x}(\varepsilon,\beta)\|}{|h|}\leq  \frac{\|\opB(\bar{x}(\varepsilon,\beta))\|}{\varepsilon}.$
combining (i), (ii) with Lemma \ref{Lemma 3.8}, we get
	\begin{align*}
		\dot{\theta}& 
		=\langle x(t)-\bar{x}(t),\dot{x}(t)\rangle-\dot{\varepsilon}(t)\langle x(t)-\bar{x}(t),\frac{\partial}{\partial \varepsilon}\bar{x}(\varepsilon(t),\beta(t))\rangle\\
		&~~~-\dot{\beta}(t)\langle x(t)-\bar{x}(t),\frac{\partial}{\partial \beta}\bar{x}(\varepsilon(t),\beta(t))\rangle\\
		&\leq -(1-\lambda(t)L(t))\|x(t)-p(t)\|^2-\varepsilon(t)\lambda(t)\|p(t)-\bar{x}(t)\|^2\\
		&~~~-\dot{\varepsilon}(t)\langle x(t)-\bar{x}(t),\frac{\partial}{\partial \varepsilon}\bar{x}(\varepsilon(t),\beta(t))\rangle\\
		&~~~-\dot{\beta}(t)\langle x(t)-\bar{x}(t),\frac{\partial}{\partial \beta}\bar{x}(\varepsilon(t),\beta(t))\rangle
	\end{align*}
%
%
	As $\Phi_{t}=\opA+V(t,\cdot)$ is $\varepsilon(t)$-strongly monotone, \eqref{eq:Phi-Monotone} shows that 
\begin{align*}
	&\lambda(t)\varepsilon(t)\|p(t)-\bar{x}(t)\|^2 \\
	&\leq \langle x(t)-p(t)+\lambda(t)[V(t,p(t))-V(t,x(t))],p(t)-\bar{x}(t) \rangle,
\end{align*}
Using Cauchy-Schwarz, and the $L(t)$-Lipschitz continuity of $V_{t}$, we obtain 
\[
\norm{p(t)-\bar{x}(t)}\leq \big(\frac{1}{\lambda(t)\varepsilon(t)}+1+\frac{1}{\eta\varepsilon(t)}+\frac{\beta(t)}{\mu\varepsilon(t)}\big)\|x(t)-p(t)\|.
 \]
it follows $\|x(t)-\bar{x}(t)\|\leq \|x(t)-p(t)\|+\|p(t)-\bar{x}(t)\|\leq a(t)\|x(t)-p(t)\|.$
For almost all $t \geq 0$, we thus get $-\|x(t)-p(t)\|^2 \leq -\frac{1}{a^2(t)}\|x(t)-\bar{x}(t)\|^2$. 
Define $\varphi \eqdef\sqrt{2\theta}$, combining with the last bound on $\dot{\theta}$, we obtain
	\begin{align*}
	\dot{\theta}(t)=\dot{\varphi}(t)\varphi(t)&\leq -\frac{1-\lambda(t)L(t)}{a^2(t)}\varphi(t)^{2}-\dot{\varepsilon}(t)\varphi(t)\frac{\|\bar{x}(t)\|}{\eps(t)}\\
	&~~~+\frac{\dot{\beta}(t)}{\varepsilon(t)}\varphi(t)\|\opB(\bar{x}(t))\|.
\end{align*}
Define $\delta(t)\eqdef\frac{1-\lambda(t)L(t)}{a^2(t)} $ and the integrating factor $\Delta(t)\eqdef\int_{0}^{t}\delta(s)\dd s$, as well as $w(t):=\norm{\bar{x}(t)}-\frac{\dot{\beta}(t)}{\dot{\varepsilon}(t)}\norm{\opB(\bar{x}(t))}$. We can then continue from the previous display with 
$
\frac{\dd}{\dd t}\big(\varphi(t)\exp(\Delta(t))\big)\leq -\frac{\dot{\varepsilon}(t)}{\varepsilon(t)}\exp(\Delta(t))w(t).
$
Integrating both sides from 0 to $t$, it immediately follows
$$
\varphi(t) \leq \exp(-\Delta(t))\bigg[\varphi(0)-\int_0^t\bigg(\exp(\Delta(s))\frac{\dot{\eps}(s)}{\eps(s)}w(s)\bigg)\dd s\bigg].
$$
If $t\mapsto \int_{0}^{t}\exp(\Delta(s))\frac{\dot{\varepsilon}(s)}{\varepsilon(s)}w(s)\dd s$ is bounded, then we immediately obtain from hypothesis (i) that $\varphi(t)\to 0$. Otherwise, we apply l'H\^{o}pital's rule to get 
\begin{align*}
\lim_{t\to\infty}& \exp(-\Delta(t))\int_{0}^{t}\exp(\Delta(s))\frac{\dot{\eps}(s)}{\eps(s)}w(s)\dd s\\
&=\lim_{t\to\infty}\frac{\exp(\Delta(t))\frac{\dot{\eps}(t)}{\eps(t)}w(t)}{\delta(t)\exp(\Delta(t))}=\lim_{t\to\infty}\frac{\frac{\dot{\eps}(t)}{\eps(t)}w(t)}{\delta(t)}.
\end{align*}
Since $\setC=\zer(\opB)$, and $\opB$ is Lipschitz, we know that  $t\mapsto\norm{\opB(\bar{x}(t))}$ is bounded. Additionally, we know from the proof of Proposition \ref{prop:asymptotics} that $\norm{\bar{x}(t)}\leq\inf\{\norm{x}:x\in\zer(\Phi)\}$. Using conditions (a) and (b), we deduce that $\varphi(t)\to 0$ and therefore $\norm{x(t)-\bar{x}(t)}\to 0$. Using Proposition \ref{prop:asymptotics}, we conclude $\|x(t)-\Pi_{\zer(\Phi)}(0)\| \leq \|x(t)-\bar{x}(t)\|+\|\bar{x}(t)-\Pi_{\zer(\Phi)}(0)\|$. Hence, $x(t)\rightarrow \Pi_{\zer(\Phi)}(0)$ as $t \rightarrow +\infty$.
\end{proof}
In the remainder of this note, we give some concrete specifications for functions $\eps(t),\lambda(t)$ and $\beta(t)$ satisfying all conditions for Theorem \ref{th:main} to hold. By hypothesis, we have 
$\liminf_{t\to\infty}(1-\lambda(t)L(t))=1-\limsup_{t\to\infty}\lambda(t)L(t)>0$. Additionally, 
$$
a(t)=2+\frac{1}{\eps(t)}\left(\frac{1}{\lambda(t)}+\frac{1}{\eta}+\frac{\beta(t)}{\mu}\right)=\frac{\lambda(t)(\eps(t)+L(t))+1}{\lambda(t)\eps(t)},
$$
which implies $a(t)=\frac{L(t)}{\eps(t)}(1+\frac{\eps(t)}{L(t)}+\frac{1}{L(t)\lambda(t)})=\scrO(\beta(t)/\eps(t))$, using that $L(t)=\scrO(\beta(t))$. This in turn implies $\delta(t)=\frac{1-\lambda(t)L(t)}{a^{2}(t)}=\scrO(\eps^{2}(t)/\beta^{2}(t))$. Hence, $\lim_{t\to\infty}\delta(t)=0$, and for obtaining $\delta\notin L^{1}(\R_{+})$ it suffices to guarantee that $\int_{0}^{\infty}\frac{\eps^{2}(t)}{\beta^{2}(t)}\dd t=\infty$. Then, 
$
\frac{\dot{\eps}(t)}{\eps(t)\delta(t)}=\frac{\dot{\eps}(t)L^{2}(t)}{\eps^{3}(t)}\frac{(1+\frac{\eps(t)}{L(t)}+\frac{1}{(L(t)\lambda(t)})^{2}}{1-\lambda(t)L(t)}\\
=\frac{\dot{\eps}(t)\beta^{2}(t)}{\eps^{3}(t)}\scrO(1).$
It therefore suffices to have $\lim_{t\to\infty}\frac{\dot{\eps}(t)\beta^{2}(t)}{\eps^{3}(t)}=0$. By a similar argument, it is easy to see that $\frac{\dot{\beta}(t)}{\eps(t)\delta(t)}=\frac{\dot{\beta}(t)\beta(t)^{2}}{\eps(t)^{3}}\scrO(1)$. Therefore, it suffices to ensure that $\frac{\dot{\beta}(t)\beta(t)^{2}}{\eps(t)^{3}}$ is bounded. Finding such functions is not too difficult, as the next remark shows. 

\begin{remark}
Assume $\eps(t)=(t+b)^{-(r+q)}, \beta(t)=(t+b)^q$, where $q,r,b> 0$ are chosen such that $r+q>0$. Then $\frac{\eps^{2}(t)}{\beta^{2}(t)}=(t+b)^{-2(r+2q)}$, and consequently we need to impose the restriction $2q+r<1/2$ to ensure that $\delta\notin L^{1}(\R_{+})$. Additionally, we compute $\frac{\dot{\eps}(t)\beta(t)^{2}}{\eps(t)^{3}}=-(r+q)(t+b)^{2r+3q-1}$. This yields the restriction $2r+3q<1$. Finally, $\frac{\dot{\beta}(t)\beta(t)^{2}}{\eps^{3}(t)}=q(t+b)^{6q+3r-1}$, and to make this a bounded sequence, we need to impose the condition $2q+r\leq\frac{1}{3}$. These conditions together span a region of feasible parameters $r,q$ which is nonempty. 
\close
\end{remark}

\section{Numerical Examples}
To illustrate the empirical performance of our method we consider the saddle point problem 
\begin{align*}
  \min_{x_{1}\in\scrX_{1}}\max_{x_{2}\in\scrX_{2}}\inner{Q x_{1},x_{2}}+\inner{b&,x_{1}}-\inner{c,x_{2}}+g\norm{x_{1}}^{2}-f\norm{x_{2}}^{2}\\
\text{s.t.: } & x:=(x_{1},x_{2})\in\setC
\end{align*}
where the coupling constraint is a polyhedron $\setC=\argmin \frac{1}{2}\norm{T_{1}x_{1}+T_{2}x_{2}-t}^{2}$. This problem is adapted from \cite{Ouyang:2021aa}. The sets $\scrX_{i}\subseteq\R^{n_{i}}$ are individual restrictions, and $T_{i}:\R^{n_{i}}\to\R^{d}$ are linear mappings. Define $\opT=[T_{1};T_{2}]$ the stacked matrix of dimensions $d\times n,n=n_{1}+n_{2}$. The vector $t\in\R^{d}$ is a given right-hand side. To embed this problem into our formulation, we define $g(x_{1},x_{2})=\frac{1}{2}\norm{T_{1}x_{1}+T_{2}x_{2}-t}^{2}$, so that $\opB(x_{1},x_{2})=\opT^{\ast}(\opT x-t)$ is just the gradient of $g$. Clearly $\setC=\zer(B)$.

In Figure 1 we can see the development of $\|x(t)-p(t)\|$ when we define $q=0.2$ and various values of the Tikhonov term $(t+1)^{(r+q)}$. However, $\|x(t)-p(t)\|$ is the fixed point residual of the operator $\res_{\lambda A}(I-\lambda V)$. According to the convergence analysis, we get $\|x(t)-p(t)\| \to 0$ as $t \to +\infty$.
\vspace{-0.4cm}
\begin{figure}[htb]
\centering
\begin{minipage}{0.50\linewidth}
\includegraphics[scale=0.11]{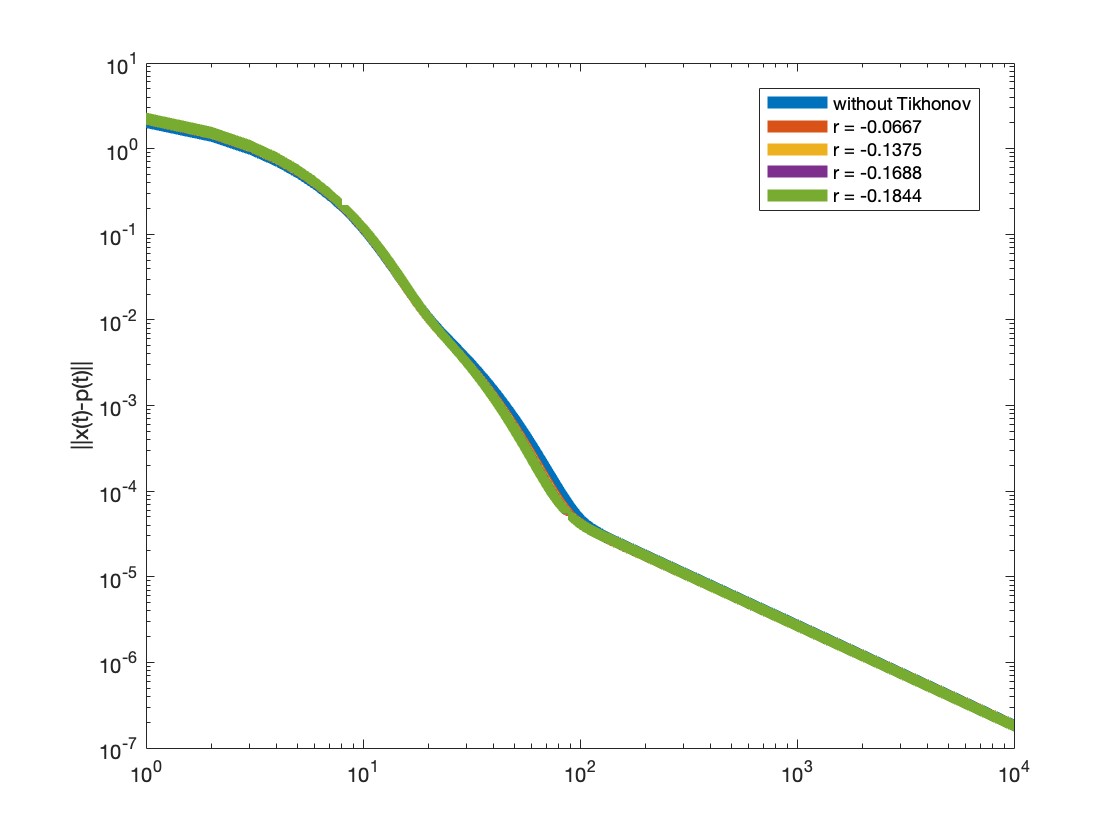}
\vspace{-0.9cm}
\caption{$\|x(t)-p(t)\|$}
\end{minipage}\hfill
\begin{minipage}{0.50\linewidth}
\centering
\includegraphics[scale=0.11]{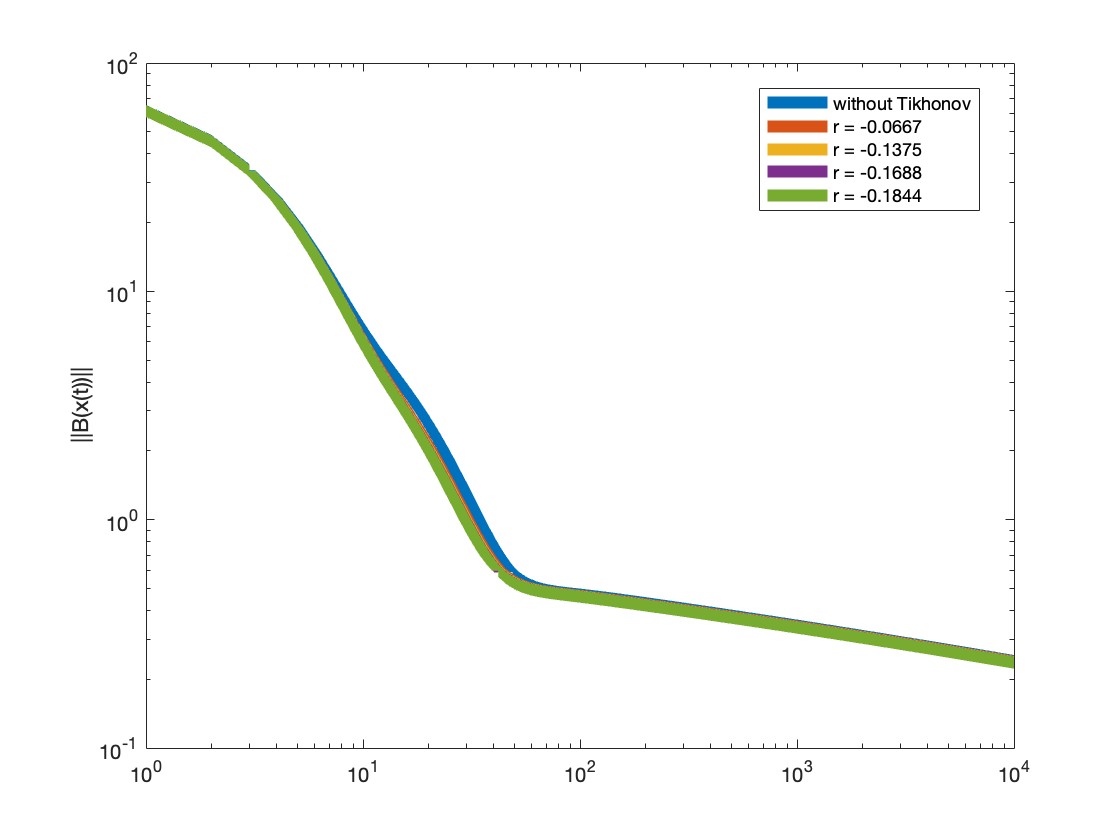}
\vspace{-0.9cm}
\caption{Feasibility Gap}
\end{minipage}
\end{figure}
\vspace{-0.42cm}


\section*{APPENDIX}

\begin{proof}[Proof of Proposition \ref{prop:asymptotics}]
To simplify notation, we set $\bar{x}_{n}\equiv\bar{x}(\eps_{n},\beta_{n})$. for all $n\geq 1$. The proof proceeds in three steps:
\begin{itemize}
\item[(i)] The sequence $(\bar{x}_{n})_{n\in\N}$ is bounded. \\
Let $z\in\zer(\Phi)$ arbitrary. Then, there exists $\xi\in\NC_{\setC}(z)$ such that $-\opD(z)-\xi\in \opA(z)$. Since $\bar{x}_{n}\in\zer(\Phi_{\eps_{n},\beta_{n}})$, we have 
$-\eps_{n}\bar{x}_{n}-\beta_{n}\opB(\bar{x}_{n})-\opD(\bar{x}_{n})\in \opA(\bar{x}_{n})$ for all $n\geq 1.$ Since $\opA$ is maximally monotone, we have for all $n\geq 1$:
\begin{equation}\label{eq:ineq1}
\inner{-\opD(z)-\xi+\eps\bar{x}_{n}+\beta_{n}\opB(\bar{x}_{n})+\opD(\bar{x}_{n}),z-\bar{x}_{n}}\geq 0 
\end{equation}
Rearranging, and using the fact that $\opB(z)=0$ (since $z\in\setC)$ as well that $B$ is a monotone operator, it follows 
$$
\eps\inner{\bar{x}_{n},z-\bar{x}_{n}}\geq \beta\inner{\opB(z)-\opB(\bar{x}_{n}),z-\bar{x}_{n}}\geq 0 
$$
By Cauchy-Schwarz, we therefore obtain $\norm{\bar{x}_{n}}\leq \norm{z}$ for all $z\in\zer(\Phi)$. It follows $\sup_{n\geq 1}\norm{\bar{x}_{n}}\leq \inf\{\norm{x}:x\in\zer(\Phi)\}$. 
 
\item[(ii)] Weak accumulation points of $(\bar{x}_{n})$ are in $\setC$. \\
Since $(\bar{x}_{n})_{n\geq 1}$ is bounded, we can extract a weakly converging subsequence $\bar{x}_{n_{j}}\wlim \bar{x}$. Let $\eps_{n_{j}},\beta_{n_{j}}$ the corresponding subsequences of the parameters sequences $(\eps_{n},\beta_{n})$. Using \eqref{eq:ineq1} and the monotonicity of $\opD$, we see 
$$
\beta_{n_{j}}\inner{\opB(\bar{x}_{n_{j}}),z-\bar{x}_{n_{j}}}\geq -\eps_{n_{j}}\inner{\bar{x}_{n_{j}},z-\bar{x}_{n_{j}}}+\inner{\xi,z-\bar{x}_{n_{j}}}.
$$
$$
\inner{\opB(\bar{x}_{n_{j})},\bar{x}_{n_{j}}-z}\leq \frac{\eps_{n_{j}}}{\beta_{n_{j}}}\inner{\bar{x}_{n_{j}},z-\bar{x}_{n_{j}}}+\frac{1}{\beta_{n_{j}}}\inner{\xi,\bar{x}_{n_{j}}-z}.
$$
$\opB$ is cocoercive and $\opB(z)=0$, which implies that there exists a $\gamma>0$ such that
\[
\inner{\opB(\bar{x}_{n_{j}}),\bar{x}_{n_{j}}-z} \geq \gamma\norm{\opB(\bar{x}_{n_{j}})-\opB(z)}^{2}=\gamma\norm{\opB(\bar{x}_{n_{j}})}^{2}.
\]
Hence, 
\[
\norm{\opB(\bar{x}_{n_{j}})}^{2}\leq\frac{\eps_{n_{j}}}{\gamma\beta_{n_{j}}}\inner{\bar{x}_{n_{j}},z-\bar{x}_{n_{j}}}+\frac{1}{\gamma\beta_{n_{j}}}\inner{\xi,\bar{x}_{n_{j}}-z}.
\]
and from $\beta_{n_{j}}\to\infty$, we conclude $\lim_{j\to\infty}\norm{\opB(\bar{x}_{n_{j}})}=0$. By continuity, $\opB\bar{x}_{n_{j}}\to\opB(\bar{x})$, which implies $\opB(\bar{x})=0$. This is equivalent to $\bar{x}\in\setC$. 
\item[(iii)] Any weak accumulation point of $(\bar{x}_{n})_{n\in\N}$ is in $\zer(\Phi)$. \\
We use the characterization of the points in $\zer(\Phi)$ provided by Fact \ref{fact:angle}. Let $(u,w)\in\gr(\Phi)$ be arbitrary. Then, there exists $\xi\in\NC_{\setC}(u)$ such that $w-\xi-\opD(u)\in\opA(u).$ Moreover, for all $n\geq 1$, 
$$
-\eps_{n}\bar{x}_{n}-\beta_{n}\opB(\bar{x}_{n})-\opD(\bar{x}_{n})\in \opA(\bar{x}_{n}).
$$
Monotonicity of $\opA$ gives 
\[
\inner{-\eps_{n}\bar{x}_{n}-\beta_{n}\opB(\bar{x}_{n})-\opD(\bar{x}_{n})-w+\xi+\opD(u),\bar{x}_{n}-u}\geq 0
\]
Using the monotonicity of $\opD$, 
\begin{align*}
\inner{w,u-\bar{x}_{n}}&\geq \eps_{n}\inner{\bar{x}_{n},\bar{x}_{n}-u}+\inner{\beta_{n}\opB(\bar{x}_{n}),\bar{x}_{n}-u}\\
&~~~+\inner{\xi,u-\bar{x}_{n}}
\end{align*}
Since $\opB$ is monotone and $\opB(u)=0$, we also have $\inner{\opB(\bar{x}_{n}),\bar{x}_{n}-u}\geq 0$. Consequently, 
\[
\inner{w,u-\bar{x}_{n}}\geq \eps_{n}\inner{\bar{x}_{n},\bar{x}_{n}-u}+\inner{\xi,u-\bar{x}_{n}}.
\]
Therefore, $\liminf_{n\to\infty}\inner{w,u-\bar{x}_{n}}\geq \inner{\xi,u-\bar{x}}\geq 0$ using that $\xi\in\NC_{C}(u)$. Using Fact \ref{fact:angle}, this means that accumulation points of $\bar{x}_{n}$ are solutions of the original problem.

\item[(iv)] $\bar{x}_{n}\to \bar{x}=\argmin\{\norm{x}:x\in\zer(\Phi)\}$.\\
Let $\bar{x}$ be any weak limit of $(\bar{x}_{n})_{n}$. We know that $\bar{x}\in\zer(\Phi)$. Step (i) of the proof shows $\norm{\bar{x}}\leq\inf\{\norm{x}:x\in\zer(\Phi)\}$. The claim follows.
\end{itemize}
\end{proof}

\begin{proof}[Proof of Proposition \ref{prop:solutionmap}]
Fix $\beta>0$ and pick $\eps_{1},\eps_{2}>0$. Let $\opD_{\eps}:=\opD+\eps\Id$ and set $z_{1}=\bar{x}(\eps_{1},\beta)$ and $z_{2}=\bar{x}(\eps_{2},\beta)$. It follows 
$-\opD_{\eps_{1}}(z_{1})-\beta \opB(z_{1})\in \opA(z_{1})$, and $-\opD_{\eps_{2}}(z_{2})-\beta\opB(z_{2})\in\opA(z_{2})$. Since $\opA$ is maximally monotone, we have 
$$
\inner{z_{1}-z_{2},-\opD_{\eps_{1}}(z_{1})-\beta\opB(z_{1})+\opD_{\eps_{2}}(z_{2})+\beta\opB(z_{2})}\geq 0.
$$
Since $\opD$ and $\opB$ are both maximally monotone, we conclude $\inner{\eps_{1}z_{1}-\eps_{2}z_{2},z_{1}-z_{2}}\leq 0$. Assume first that $\eps_{2}>\eps_{1}$. Then
\[
0\geq \inner{\eps_{1}z_{1}-\eps_{2}z_{2},z_{1}-z_{2}}=\eps_{1}\norm{z_{1}-z_{2}}^{2}+(\eps_{1}-\eps_{2})\inner{z_{2},z_{1}-z_{2}}, 
\]
which means $(\eps_{2}-\eps_{1})\inner{z_{2},z_{1}-z_{2}}\geq\eps_{1}\norm{z_{1}-z_{2}}^{2}$. By Cauchy-Schwarz, 
$(\eps_{2}-\eps_{1})\norm{z_{2}}\cdot\norm{z_{1}-z_{2}}\geq \eps_{1}\norm{z_{1}-z_{2}}^{2},$ so that $\norm{z_{2}-z_{1}}\leq\frac{\eps_{2}-\eps_{1}}{\eps_{1}}\norm{z_{2}}.$ Next, assuming $\eps_{1}>\eps_{2}$. Then, interchanging the labels in the above inequality, we get 
$\norm{z_{2}-z_{1}}\leq\frac{\eps_{1}-\eps_{2}}{\eps_{2}}\norm{z_{1}}.$ Hence, $\norm{\bar{x}(\eps_{1},\beta)-\bar{x}(\eps_{2},\beta)}\leq \frac{\abs{\eps_{2}-\eps_{1}}}{\max\{\eps_{1},\eps_{2}\}}\max\{\norm{z_{1}},\norm{z_{2}}\}$. This shows that $\eps\mapsto \bar{x}(\eps,\beta)$ is locally Lipschitz. 

Now, fix $\eps>0$ and let $\beta_{1},\beta_{2}>0$. Denote $z_{1}=\bar{x}(\eps,\beta_{1})$ and $z_{2}=\bar{x}(\eps,\beta_{2})$. By definition, we have 
$$
-D_{\eps}z_{1}-\beta_{1} Bz_{1}\in Az_{1},\text{ and }-D_{\eps}z_{2}-\beta_{2} Bz_{2}\in Az_{2}.
$$
It follows $\beta_{2}\inner{\opB(z_{2}),z_{1}-z_{2}}-\beta_{1}\inner{\opB(z_{1}),z_{1}-z_{2}}\geq\eps\norm{z_{1}-z_{2}}^{2}.$ Assume that $\beta_{2}>\beta_{1}$. Then  
$(\beta_{2}-\beta_{1})\inner{\opB(z_{1}),z_{1}-z_{2}}+\beta_{2}\inner{\opB(z_{2})-\opB(z_{1}),z_{1}-z_{2}}\geq \eps\norm{z_{1}-z_{2}}^{2}$. Using the monotonicity of $\opB$, we conclude 
$$
\norm{z_{1}-z_{2}}\leq \frac{\beta_{2}-\beta_{2}}{\eps}\norm{\opB (z_{1})}.
$$
If $\beta_{1}>\beta_{2}$, we repeat the above computation, and obtain $\norm{z_{1}-z_{2}}\leq \frac{\beta_{1}-\beta_{2}}{\eps}\norm{\opB(z_{2})}.$ This yields $\norm{z_{1}-z_{2}}\leq \frac{\abs{\beta_{1}-\beta_{2}}}{\eps}\max\{\norm{\opB(z_{1})},\norm{\opB(z_{2})}\}$, which shows that $\beta\mapsto\bar{x}(\eps,\beta)$ is locally Lipschitz, for all $\eps>0$. 

Next, we show the Lipschitz continuity of the bivariate map $(\eps,\beta)\mapsto\bar{x}(\eps,\beta)$. Let $t_{1}\eqdef (\eps_{1},\beta_{1})$ and $t_{2}\eqdef(\eps_{2},\beta_{2})$ with corresponding solutions $\bar{x}(t_{1})$ and $\bar{x}(t_{2})$. By definition of these points, we have 
$
-V_{(\eps_{1},\beta_{1})}\bar{x}(t_{1})\in \opA\bar{x}(t_{1}),\text{ and }-V_{(\eps_{2},\beta_{2})}(\bar{x}(t_{2}))\in \opA\bar{x}(t_{2}).
$
Hence, 
$$
\inner{\opD_{\eps_{2}}\bar{x}(t_{2})+\beta_{2}\opB\bar{x}(t_{2})-\opD_{\eps_{1}}\bar{x}(t_{1})-\beta_{1}\opB\bar{x}(t_{1}),\bar{x}(t_{1})-\bar{x}(t_{2})}\geq 0.
$$
Rearranging, we obtain 
\begin{align*}
\inner{\eps_{2}\bar{x}(t_{2})-\eps_{1}\bar{x}(t_{1}),\bar{x}(t_{1})-\bar{x}(t_{2})}\geq&\beta_{1}\inner{\opB\bar{x}(t_{1}),\bar{x}(t_{1})-\bar{x}(t_{2})}\\
&+\beta_{2}\inner{\opB\bar{x}(t_{2}),\bar{x}(t_{2})-\bar{x}(t_{1})}.
\end{align*}
Then,
\begin{align*}
\eps_{1}\norm{\bar{x}(t_{1})-\bar{x}(t_{2})}^{2}&\leq \abs{\beta_{2}-\beta_{1}}\norm{\opB\bar{x}(t_{1})}\cdot\norm{\bar{x}(t_{2})-\bar{x}(t_{1})}\\
&~~~+\abs{\eps_{2}-\eps_{1}}\norm{\bar{x}(t_{2})}\cdot\norm{\bar{x}(t_{2})-\bar{x}(t_{1})}.
\end{align*}
It follows 
$$
\norm{\bar{x}(t_{2})-\bar{x}(t_{1})}\leq \frac{\abs{\beta_{2}-\beta_{1}}}{\eps_{1}}\norm{\opB\bar{x}(t_{1})}+\frac{\abs{\eps_{2}-\eps_{1}}}{\eps_{1}}\norm{\bar{x}(t_{2})}.
$$
From the proof of Step (i) of the proof of Proposition \ref{prop:asymptotics}, we deduce that $\norm{\bar{x}(\eps,\beta)}\leq\inf\{\norm{x}:x\in\zer(\Phi)\}\eqdef \ca$. Hence, defining 
$\ell\eqdef \max\{\sup_{x\in\B(0,\ca)}\norm{\opB x},\ca\}$, the claim follows.
\end{proof}


\end{document}

%% file: preamble.tex
\newcommand{\eps}{\varepsilon}
\DeclareMathOperator*{\argmin}{argmin}

\DeclareMathOperator{\zer}{zer}

\DeclareMathOperator{\dif}{d\!}

\DeclareMathOperator{\dom}{dom}

\DeclareMathOperator{\gr}{graph}
\DeclareMathOperator{\Int}{int}

\DeclareMathOperator{\res}{\mathsf{J}}

\DeclareMathOperator{\Id}{Id}

\newcommand{\ca}{\mathtt{a}}

\newcommand{\B}{\mathbb{B}}

\renewcommand{\iff}{\Leftrightarrow}

\renewcommand{\emptyset}{\varnothing}
\newcommand{\eqdef}{\triangleq}
\newcommand{\wlim}{\rightharpoonup}

\newcommand{\scrG}{\mathcal{G}}
\newcommand{\scrH}{\mathcal{H}}

\newcommand{\scrO}{\mathcal{O}}

\newcommand{\scrX}{\mathcal{X}}

\newcommand{\setC}{\mathsf{C}}

\newcommand{\M}{{\mathsf{M}}}

\newcommand{\R}{\mathbb{R}}

\newcommand{\N}{\mathbb{N}}

\DeclareMathOperator{\NC}{\mathsf{N}}

\newcommand{\opA}{\mathsf{A}}
\newcommand{\opB}{\mathsf{B}}

\newcommand{\opD}{\mathsf{D}}
\newcommand{\opT}{\mathsf{T}}

\newtheorem{theorem}{Theorem}

\newtheorem{lemma}[theorem]{Lemma}
\newtheorem{proposition}[theorem]{Proposition}
\newtheorem{fact}[theorem]{Fact}

\newtheorem{assumption}{Assumption}
\newcommand{\close}{\hfill{\footnotesize$\Diamond$}}

\newtheorem{remark}{Remark}
\newtheorem{example}{Example}

\DeclarePairedDelimiter{\inner}{\langle}{\rangle}